\documentclass[12pt,twoside, final]{amsart}

\usepackage{amsmath,amsthm,amscd,amsfonts,amssymb,enumerate}
\usepackage{graphicx}
\usepackage{color}
\usepackage[colorlinks]{hyperref}
\usepackage{amsfonts,amssymb,amscd,amsmath,enumerate,url,verbatim}
 \usepackage[dvips]{epsfig}
 \usepackage[none]{hyphenat}
\usepackage{amsmath,amssymb,amsfonts,enumerate,amsthm}
 \usepackage{amsgen, amstext,amsbsy,amsopn, amsthm, amsfonts,amssymb,amscd,amsmath,euscript,enumerate,url,verbatim,calc,xypic}
 \usepackage{latexsym}
 \usepackage{graphics}
 \usepackage{color}
\numberwithin{equation}{section}
\newtheorem{thm}{Theorem}[section]
\newtheorem{lem}[thm]{Lemma}
\newtheorem{cor}[thm]{Corollary}
\newtheorem{con}[thm]{Conjecture}

\newtheorem{qu}[thm]{Question}

\newtheorem{nota rem}[thm]{Notation and Remark}

\newtheorem{defen rem}[thm]{Definition and Remark}

\def\hei{\operatorname{ht}}

\def\pd{\operatorname{pd} }
 \def\Ext{\operatorname{Ext}}
\def\Hom{\operatorname{Hom}}
\def\Spec{\operatorname{Spec}}

\def\grade{\operatorname{grade}}
 
\def\cd{\operatorname{cd}}
\def\id{\operatorname{id}}

\def\Supp{\operatorname{Supp}}
\def\var{\operatorname{Var}}

 \def\dim{\operatorname{dim}}
\newcommand{\p}{\mathfrak{p}}

\newcommand{\fa}{\mathfrak{a}}

\newcommand{\m}{\mathfrak{m}}

\newcommand{\Z}{\mathbb{Z}}

\newcommand{\NN}{\mathbb{N}}

\begin{document}
\bibliographystyle{amsplain}

\author{M. Jahangiri}
\address{ Department of Mathematics, Faculty of Mathematical Sciences and Computer, Kharazmi University, Tehran,
Iran.}
\email{ jahangiri@khu.ac.ir\\
jahangiri.maryam@gmail.com}

 \author{R. Ahangari Maleki}
\email{ahangari.rasoul@gmail.com}
 

\title[  Bass numbers of local cohomology modules ...]
{  Bass numbers of   local cohomology modules AT THE FIRST AND LAST NON-VANISHING LEVELs}

\begin{abstract}
Let $R $ be a commutative Noetherian ring, $\fa$ be an ideal of $R$   and $M$  be a finitely generated    $R$-module.  In this paper, we study the Bass numbers  
$\{\mu^i(\p, H^j_{\fa}(M))\} $   of local cohomology modules with respect to an ideal $\p\in \Spec(R)$ in each of the following cases:
\begin{enumerate}
    \item $i\in \{ 0, 1, 2\}$ and $j= \grade_{\fa}(M),$ 
  \item  $R$ is regular and $i\in\{ \hei(\p),  \hei(\p)- 1\}$ and $j= \cd_{\fa}(M)$, the cohomological dimension of $M$ with respect to $\fa$.
  \end{enumerate}
  
\end{abstract}
\maketitle

\section{introduction}

Throughout the paper, let $R$ be a non-trivial commutative Noetherian ring with identity, $\fa$   an ideal of $R$ and $M$ be a non-zero finitely generated $R$-module.     $\Z$, $\NN_0$ and $\NN $, respectively, stand for the set of   integers, non-negative integers and positive integers.

For any $i\in \NN_0$, $H^i_{\fa}(-)$ denotes the $i$-th local cohomology functor with respect to $\fa$ which is the $i$-th right derived functor of the $\fa$-torsion functor $\Gamma_{\fa}(-)$. Recall that for any $R$-module $X$,
 \[\Gamma_{\fa}(X):= \cup_{t\in \NN} (0:_X \fa^t)= \{x\in \NN| \exists t\in \NN, \fa^t x=0\}.\]
 For the basic properties of local cohomology modules, we refer the reader to \cite{bsh}.  
Vanishing and finiteness properties of local cohomology modules  are among the central problems in this theory and have attracted considerable attention, see for example \cite{al, dnt, hu1, l}  and \cite{o}.
In \cite{hu1}, Huneke conjectured that:
 \begin{con}\label{conj}
     If $(R, \m)$ is a regular local ring, then for any $\p\in \Spec(R)$ the Bass numbers \[\mu^i(\p, H^j_{\fa}(R)):= \dim_{\kappa(\p)}(\Ext^i_{R_{\p}}(\kappa(\p), H^j_{\fa}(R)_{\p}))\] of local cohomology modules are finite for all $i, j \in \NN_0.$
 \end{con}
 Where,  $\kappa(\p):= \tfrac{R_{\p}}{\p R_{\p}}$ is the residue field of $R_{\p}$. Obviously, the Bass numbers $\mu^i(\p, H^j_{\fa}(R))$ are finite if $H^j_{\fa}(R)$ is finitely generated. But, local cohomology modules are hardly ever finitely generated!
In the case where $(R, \m)$ is a regular local ring containing a field the above conjecture has an affirmative answer. Actually,   Huneke and Sharp (\cite{hu}), in the case of positive characteristic, and Lyubeznik (\cite{lyu2})  showed that in this case $\mu^i(\m, H^j_{\fa}(R))< \infty$ for all   $i, j \in \NN_0.$

Note that if $R$ is not regular, the Bass numbers  $\mu^i(\p, H^j_{\fa}(R))$ may be infinite, as shown by an example of Hartshorne (\cite{ha}). He proved that  $\mu^0((x, y, z, w), H^j_{(x, y)}(\tfrac{K[x, y, z, w]}{(xz- yw)}))= \infty$, where $K$ is a field. 

Also, Bahmanpour et  al. in \cite{bns}   showed that if $(R, \m)$ is a regular local ring then $\mu^0(\p, H^2_{\fa}(R))$ and $\mu^1(\p, H^2_{\fa}(R))$ are finite for all $\p\in \Spec(R)$ and that,  when $\dim(R)\leq 3$ then  the Bass numbers of $ H^j_{\fa}(M)$ are finite, for any ideal $\fa$ of $R$ and   $i\in \NN_0$, see \cite{ba}. 

One can see \cite{al, he}  and \cite{yan} for research on the Bass numbers of local cohomology modules under different conditions.

 In this paper, we provide partial answers to the Conjecture \ref{conj} and we consider the Bass numbers of $H^j_{\fa}(M)$
  where $j$ is the first and last non-vanishing levels of local cohomology modules of $M$ with respect to $\fa$. More precisely, we prove the following statements.
Here, $ \cd_{\fa}(M)$ denotes the cohomological dimension of $M$ with respect to $\fa$, the
last non-vanishing level of local cohomology modules of $M $ with respect to $\fa$, see \ref{defcd} for definition.
  \begin{thm}
  Assume that $R$ is regular, $\p\in \Supp_R(H^c_{\fa}(M))$ and set $d:= \hei(\p)$ and $c:= \cd_{\fa}(M)$. Then the following statements hold.
    \begin{enumerate}
        \item $\mu^d(\p, H^c_{\fa}(M))= \mu^{d+ c}(\p, M)< \infty$.\label{d,c}
        \item $\mu^{d- 1}(\p, H^c_{\fa}(M))\leq  \mu^{d+ c- 1}(\p, M)< \infty$. The equality holds if $\mu^d(\p, H^{c- 1}_{\fa}(M))= 0$.\label{d-1,c}
        \item If $\mu^d(\p, H^{c- 1}_{\fa}(M))= 0$, then $\mu^{d- 2}(\p, H^c_{\fa}(M))\leq \mu^{d+ c- 2}(\p, M)< \infty$.\label{d,c-1}
         \end{enumerate}
 \end{thm}
 The above theorem yields the following result concerning the injective dimension of local cohomology modules.

 \begin{thm}
     Let $R$ be a regular ring and  $c:= \cd_{\fa}(M)$. Then the following statements hold.
 \begin{enumerate}
       \item If $c> 0$, then $\id_{R}(H^c_{\fa}(M))< \dim(R)$ and $H^c_{\fa}(M)$ is not finitely generated.
        \item If $c> 1$, then $\id_{R}(H^c_{\fa}(M))< \dim(R)- 1$.
        \item If $c> 2$ and   $\mu^{\hei(\p)}(\p, H^{c- 1}_{\fa}(M))= 0$ for any $\p\in \Spec(R)$, then $\id_{R}(H^c_{\fa}(M))< \dim(R)- 2$.
 \end{enumerate}
     
 \end{thm}
Moreover, if $\grade_{\fa}(M)$ denotes the length of a maximal $M$-regular sequence contained in $\fa$, then we establish the following result concerning the Bass numbers  of $ H^{\grade_{\fa}(M)}_{\fa}(M)$. Note that, in view of \cite[Theorem 6. 2. 7]{bsh}, $\grade_{\fa}(M)$ is the first non-vanishing level of local cohomology modules of $M$ with respect to $\fa$. 

 The following theorem, also,  may be regarded as a generalization of  \cite[Theorem 2. 1]{bns}, to the non regular case, where the grade is grater than 1.
 \begin{thm}
 Let $\p\in \var(\fa)$ and   $g\leq \grade_{\fa}(M)$. Then the following statements hold.
\begin{enumerate}
    \item $\mu^0(\p, H^g_{\fa}(M))= \mu^g(\p, M)< \infty$.
    \item $\mu^1(\p, H^g_{\fa}(M))\leq \mu^{g+ 1}(\p, M)$ and if $\mu^0(\p, H^{g+ 1}_{\fa}(M))= 0$, then the equality holds.
    \item If $\mu^2(\p, H^g_{\fa}(M))= 0$, then $\mu^0(\p, H^{g+ 1}_{\fa}(M)) \leq \mu^{g+ 1}(\p, M)< \infty$.
    \item If $\mu^3(\p, H^g_{\fa}(M))= 0$, then $\mu^1(\p, H^{g+ 1}_{\fa}(M)) \leq \mu^{g+ 2}(\p, M)< \infty$.
\end{enumerate}    
 \end{thm}
From the above theorem, we deduce that $H^{\grade_{\fa}(M)}_{\fa}(M)$ has ﬁnite Goldie dimension.

It is worth mentioning  that, in \cite{le}, the author computes the Bass numbers $\mu^0(\p, H^g_{\fa}(R))$ and $\mu^1(\p, H^g_{\fa}(R))$ in the case where  $R$ is a regular local ring containing a field, $\fa$ is  an ideal with $g= \grade_{\fa}(R)$, such that
 $\hei(\p)= g$ for every minimal prime $\p$ of $\fa$.
 
We keep the assumptions and notations introduced in the Introduction throughout the paper.
\section{Results}
 
The following version of the Grothendieck spectral sequences will be used in the paper.
\begin{lem} (\cite[Theorem 10. 47]{r}) \label{rot10.47}
      Let $\mathcal{A}\stackrel{\mathcal{F}}{ \rightarrow} \mathcal{B} \stackrel{\mathcal{G}}{ \rightarrow}\mathcal{C}$
  be covariant additive functors, where $\mathcal{A}$, $\mathcal{B}$, and $\mathcal{C}$ are abelian categories with enough injectives. Assume that $\mathcal{F}$ is left exact and that $\mathcal{G}(E)$ is right $\mathcal{F}$-acyclic for every
injective object $E$ in $\mathcal{A}$. Then, for every object $A$ in $\mathcal{A}$, there is a third quadrant spectral sequences with
\[E_2^{i, j}= R^i\mathcal{F}(R^j\mathcal{G}(A))\underset{i}{\Rightarrow}R^{i+ j}\mathcal{F}\mathcal{G} (A).\]
\end{lem}

 The following convergence of spectral sequences plays a crucial role in the paper.
\begin{lem}\label{spec}
 Assume that  $(R, \m)$ is local and $N$ is an $R$-module. Then there exists the following   convergence of spectral sequences.
\begin{equation}\label{1}
     E_2^{i, j}= \Ext^i_{R }(\tfrac{R}{\m}, H^j_{\fa}(N))\underset{i}{\Rightarrow}\Ext^{i+ j}_{R}(\tfrac{R}{\m}, N).
\end{equation}
   
 \end{lem}
\begin{proof}
  
Let $E$ be an  injective $R$-module. Then, by \cite[Proposition  2. 1. 4]{bsh}, $\Gamma_{\fa}(E)$ is an  injective $R$-module, too.   
Also, we have the following  isomorphisms of   $R$-modules
\begin{align}\label{a}
\Hom_R(\tfrac{R}{\m}, \Gamma_{\fa}(N))&\cong \Hom_R(\tfrac{R}{\m},\underset{ \overset {\longrightarrow}{t\in \mathbb{N}}}{ \lim }\,\,   \Hom _{R}(\tfrac{R}{\fa^t}, N) && \text{(by \cite[Theorem 1. 3. 8]{bsh})} \\
&\cong \underset{ \overset {\longrightarrow}{t\in \mathbb{N}}}{ \lim }\,\,\Hom_R(\tfrac{R}{\m}, \Hom _{R}(\tfrac{R}{\fa^t}, N)) && \text{(by  (\cite[24. 10]{wis}) } \nonumber\\
&\cong \underset{ \overset {\longrightarrow}{t\in \mathbb{N}}}{ \lim }\,\, \Hom_R(\tfrac{R}{\m}\otimes_R \tfrac{R}{\fa^t}, N)  \nonumber\\
&=  \Hom_R(\tfrac{R}{\m}, N).\nonumber
\end{align}
The above isomorphisms are functorial, i.e.
\begin{equation}\label{wwww}
\Hom_R(\tfrac{R}{\m}, \Gamma_{\fa}(-))\cong  \Hom_{R}(\tfrac{R}{\m}, -): 
  \mathfrak{C} _R\rightarrow 
  \mathfrak{C}_R,
\end{equation}
where  $\mathfrak{C} _R$   denotes the category of all $R$-modules and   $R$-homomorphisms. Now, Lemma \ref{rot10.47} completes the proof.
 \end{proof}
In the following theorem, we consider Bass numbers $\mu^i(\p, H^g_{\fa}(M))$ where $g$ is close to the $\grade_{\fa}(M)$, the length of a maximal $M$-regular sequence contained in $\fa$. 

Recall that  if $X$ is an $R$-module and $\p\in \Spec(R)$, then for any $i\in \NN_0$, the $i$-th Bass number of $X$ with respect to $\p$ is defined by
\begin{equation}\label{Bass}
    \mu^i(\p, X):= \dim_{\kappa(p)}(\Ext^i_{R_{\p}}(\kappa(\p), X_{\p})).
\end{equation}
The really important property of the
$i$-th Bass number of $X$ with respect to $\p$  is that it gives, when the (uniquely determined) $i$-th term in the minimal injective resolution of $X$ is expressed as a direct sum of indecomposable injective $R$-modules, the number of those direct summands which
are isomorphic to $E_R(\tfrac{R}{\p})$, the injective hull of the $R$-module $\tfrac{R}{\p}$.

Here, $\var(\fa)$ is used to denote the set $ \{\p\in\Spec(R)\mid \fa\subseteq \p\}$.
 
\begin{thm}\label{grade}
Let $\p\in \var(\fa)$ and   $g\leq \grade_{\fa}(M)$. Then the following statements hold.
\begin{enumerate}
    \item $\mu^0(\p, H^g_{\fa}(M))= \mu^g(\p, M)< \infty$.
    \item $\mu^1(\p, H^g_{\fa}(M))\leq \mu^{g+ 1}(\p, M)$ and if $\mu^0(\p, H^{g+ 1}_{\fa}(M))= 0$, then the equality holds.
    \item If $\mu^2(\p, H^g_{\fa}(M))= 0$, then $\mu^0(\p, H^{g+ 1}_{\fa}(M)) \leq \mu^{g+ 1}(\p, M)< \infty$.
    \item If $\mu^3(\p, H^g_{\fa}(M))= 0$, then $\mu^1(\p, H^{g+ 1}_{\fa}(M)) \leq \mu^{g+ 2}(\p, M)< \infty$.
\end{enumerate}
    \end{thm}

\begin{proof}
First of all note that, $\grade_{\fa}(M)\leq \grade_{\fa R_{\p}}(M_{\p})$ and since $\p\in \var(\fa)$, by  the flat base change theorem (\cite[Theorem 4. 3. 2]{bsh}), after replacing $R$ with $R_{\p}$ we may assume that $(R, \m)$ is local and $\p= \m$. Now, consider the convergence of spectral sequences \ref{1},
\[ E_2^{i, j}= \Ext^i_{R }(\tfrac{R}{\m}, H^j_{\fa}(M))\underset{i}{\Rightarrow}\Ext^{i+ j}_{R}(\tfrac{R}{\m}, M).\]
 \begin{enumerate}
     \item Since $g\leq \grade_{\fa}(M)$, by \cite[Theorem 6. 2. 7]{bsh},
     \begin{equation}\label{zero}
       0= \Ext^i_{R }(\tfrac{R}{\m}, H^j_{\fa}(M))= E_2^{i, j}=  E_{\infty}^{i, j} \,\,\,\,\text{for all}\,\,i\in \NN_0\,\,\text{and all}\,\,j< g.
     \end{equation}
     Also, there exists a filtration 
 \[0\subseteq \Phi_{g} \subseteq\cdots\subseteq \Phi_{   1} \subseteq\Phi_{0}= \Ext^{g}_{R}(\tfrac{R}{\m}, M),\]
 of submodules of $\Ext^{g}_{R}(\tfrac{R}{\m}, M)$ such that 
 \begin{equation*}
     E_{\infty}^{i, j}\cong \tfrac{\Phi_{i}}{\Phi_{i+ 1}},\,\,\text{for all}\,\,i, j\in \NN_0\,\,\text{with}\,\,i+ j= g.
 \end{equation*}
 Now, \ref{zero} implies that 
 \[\Hom_{R }(\tfrac{R}{\m}, H^g_{\fa}(M))=  E_2^{0, g}= E_{\infty}^{0, g}\cong  \Ext^{g}_{R}(\tfrac{R}{\m}, M).\]
 \item Using \ref{zero} we have 
 \[\Ext^1_{R }(\tfrac{R}{\m}, H^g_{\fa}(M))= E_2^{1, g}=  E_{\infty}^{1, g}.\]
 Furthermore, there exists a filtration 
 \begin{equation}\label{f1}
     0\subseteq \Phi_{g+ 1} \subseteq\cdots\subseteq \Phi_{   2} \subseteq\Phi_{   1} \subseteq\Phi_{0}= \Ext^{g+ 1}_{R}(\tfrac{R}{\m}, M),
 \end{equation}
 of submodules of $\Ext^{g+ 1}_{R}(\tfrac{R}{\m}, M)$ such that 
 \begin{equation*}
     E_{\infty}^{i, j}\cong \tfrac{\Phi_{i}}{\Phi_{i+ 1}},\,\,\text{for all}\,\,i, j\in \NN_0\,\,\text{with}\,\,i+ j= g+ 1.
 \end{equation*}
 \ref{zero} implies that the filtration \ref{f1} is actually of the form
 \[ 0= \Phi_{g+ 1} =\cdots= \Phi_{   2} \subseteq\Phi_{   1} \subseteq\Phi_{0}= \Ext^{g+ 1}_{R}(\tfrac{R}{\m}, M).\]
 Therefore,
\begin{equation}\label{2}
    \Ext^1_{R }(\tfrac{R}{\m}, H^g_{\fa}(M))\cong \Phi_{   1},
\end{equation}
 is a submodule of $\Ext^{g+ 1}_{R}(\tfrac{R}{\m}, M)$. If $\mu^0(\p, H^{g+ 1}_{\fa}(M))= 0$ then 
 \[0= E_{\infty}^{0, g+ 1}= \tfrac{\Ext^{g+ 1}_{R}(\tfrac{R}{\m}, M)}{\Phi_{  1}}.\]
 Hence, using \ref{2} we have 
 \[ \Ext^1_{R }(\tfrac{R}{\m}, H^g_{\fa}(M))\cong \Ext^{g+ 1}_{R}(\tfrac{R}{\m}, M),\]
 and the result follows.
 
 \item Assume that $\mu^2(\p, H^g_{\fa}(M))= 0$. Then $E_2^{2, g}= 0$ and by \ref{zero} we deduce that 
 \[\Hom_{R }(\tfrac{R}{\m}, H^{g+ 1}_{\fa}(M))= E_2^{0, g+ 1}= E_{\infty}^{0, g+ 1}.\]
 Using the concept of convergence of spectral sequences, as used in parts (1) and (2),  the latter module is a subquotient of $\Ext^{g+ 2}_{R}(\tfrac{R}{\m}, M)$. This proves the claim.
 \item Let $\mu^3(\p, H^g_{\fa}(M))= 0$. Then $E_2^{3, g}= 0$ and by \ref{zero} we have 
 \[\Ext^1_{R }(\tfrac{R}{\m}, H^{g+ 1}_{\fa}(M))= E_2^{1, g+ 1}= E_{\infty}^{1, g+ 1}.\] 
 Therefore, using similar argument as in the above parts,  $\Ext^1_{R }(\tfrac{R}{\m}, H^{g+ 1}_{\fa}(M))$ is a subquotient of $\Ext^{g+ 2}_{R}(\tfrac{R}{\m}, M)$. In other words, $$\mu^1(\m, H^{g+ 1}_{\fa}(M)) \leq \mu^{g+ 2}(\m, M)< \infty.$$
 \end{enumerate}
\end{proof}
Recall that an $R$-module $X$ is said to have ﬁnite Goldie dimension
if $X$ does not contain an inﬁnite direct sum of non-zero submodules, or equivalently the
injective hull $E_R(X)$ of $X$ decompose as a ﬁnite direct sum of indecomposable (injective)
submodules.

 Note that by \cite[Theorem 2. 2]{lash}, $H^{\grade_{\fa}(M)}_{\fa}(M)$ has finitely many associated prime ideals. The above theorem now yields the following corollary.
\begin{cor}
    $H^{\grade_{\fa}(M)}_{\fa}(M)$ has ﬁnite Goldie dimension.
\end{cor}
 
The  cohomological dimension
  of an $R$-module $X$ with respect   $\fa$ is defined by
\begin{equation}\label{defcd}
    \cd_{\fa}(X):= \sup\{i\in \NN_0\arrowvert H^i_{\fa}(X)\neq 0\}.
\end{equation}
Note that, in view of \cite[Corollary 3. 3. 3]{bsh},   $\cd_{\fa}(X)< \infty.$
The cohomological dimension, the last non-vanishing level of local cohomology, is a significant invariant in this theory and has attracted considerable attention; see, for example, \cite{dnt, l}  and \cite{o}.

 In the following theorem, we study the Bass numbers  $\mu^i(\p, H^j_{\fa}(M))$ in the case where $R$ is regular and $j\in \{\cd_{\fa}(M), \cd_{\fa}(M)- 1\}$.
\begin{thm}\label{reg}
    Assume that $R$ is regular, $\p\in \Supp_R(H^c_{\fa}(M))$ and set $d:= \hei(\p)$ and $c:= \cd_{\fa}(M)$. Then the following statements hold.
    \begin{enumerate}
        \item $\mu^d(\p, H^c_{\fa}(M))= \mu^{d+ c}(\p, M)< \infty$.\label{d,c}
        \item $\mu^{d- 1}(\p, H^c_{\fa}(M))\leq  \mu^{d+ c- 1}(\p, M)< \infty$.   The equality holds if $\mu^d(\p, H^{c- 1}_{\fa}(M))= 0$.\label{d-1,c}
        \item If $\mu^d(\p, H^{c- 1}_{\fa}(M))= 0$ then $\mu^{d- 2}(\p, H^c_{\fa}(M))\leq \mu^{d+ c- 2}(\p, M)< \infty$.\label{d,c-1}
           \end{enumerate}
\end{thm}
\begin{proof}
Let $\p\in \Supp_R(H^c_{\fa}(M))$. Then, by the Flat Base Change Theorem (\cite[Theorem 4. 3. 2]{bsh}), $\cd_{\fa R_{\p}}(M_{\p})= \cd_{\fa}(M)= c$. Therefore, we may assume  that $(R, \m)$ is a regular local ring of dimension $d$ and $\p= \m$. 
    \begin{enumerate}
        \item Consider the following convergence of spectral sequences (\ref{1})
         \[E_2^{i, j}= \Ext^i_{R }(\tfrac{R}{\m}, H^j_{\fa}(M))\underset{i}{\Rightarrow}\Ext^{i+ j}_{R}(\tfrac{R}{\m}, M). \] 
      Since $R$ is regular of dimension $d$ and $c= \cd_{\fa}(M)$, we have
\begin{equation}\label{3}
  E_{\infty}^{i, j}= E_2^{i, j}= 0\,\,\text{ for all}\,\,i, j \in \NN_0\,\, \text{with}\,\,i> d\,\,\text{ or}\,\,j> c. 
\end{equation}
 Therefore, one can deduce that 
 \[\Ext^d_{R }(\tfrac{R}{\m}, H^c_{\fa}(M))= E_2^{d, c}= E_{\infty}^{d, c}\cong \Ext^{d+ c}_{R}(\tfrac{R}{\m}, M),\]
 this proves part (1).
 \item By \ref{3}, 
 \begin{equation}\label{7}
     \Ext^{d- 1}_{R }(\tfrac{R}{\m}, H^c_{\fa}(M))= E_{2}^{d- 1, c}= E_{\infty}^{d- 1, c}.
 \end{equation}
 Also, there exists a filtration 
 \begin{equation}\label{f2}
     0\subseteq \Phi_{d+ c- 1} \subseteq\cdots\subseteq \Phi_{   2} \subseteq\Phi_{   1} \subseteq\Phi_{0}= \Ext^{d+ c- 1}_{R}(\tfrac{R}{\m}, M),
 \end{equation}
 of submodules of $\Ext^{d+ c- 1}_{R}(\tfrac{R}{\m}, M)$ such that 
 \begin{equation*}
     E_{\infty}^{i, j}\cong \tfrac{\Phi_{i}}{\Phi_{i+ 1}},\,\,\text{for all}\,\,i, j\in \NN_0\,\,\text{with}\,\,i+ j= d+ c- 1.
 \end{equation*}
 Using \ref{3}, the above filtration is actually of the form
 \begin{equation} \label{6}
     0= \Phi_{d+ c- 1} =\cdots= \Phi_{   d+ 1} \subseteq\Phi_{  d} \subseteq\Phi_{d- 1} =\cdots= \Phi_{0}= \Ext^{d+ c- 1}_{R}(\tfrac{R}{\m}, M).
 \end{equation}
 Therefore, by \ref{7}, 
 \[\Ext^{d- 1}_{R }(\tfrac{R}{\m}, H^c_{\fa}(M))\cong \tfrac{\Ext^{d+ c- 1}_{R}(\tfrac{R}{\m}, M)}{\phi_{d }}\]
 is a subquotient of 
 $\Ext^{d+ c- 1}_{R}(\tfrac{R}{\m}, M),$ hence $\mu^{d- 1}(\m, H^c_{\fa}(M))\leq  \mu^{d+ c- 1}(\m, M)$.
 If $\mu^d(\p, H^{c- 1}_{\fa}(M))= 0$ then $E_{\infty}^{d, c- 1}= E_2^{d, c- 1}= 0$ and in the filtration 
\ref{6}, $0= \Phi_{   d+ 1}= \Phi_{  d}$. Hence, $\Ext^{d- 1}_{R }(\tfrac{R}{\m}, H^c_{\fa}(M))\cong \Ext^{d+ c- 1}_{R}(\tfrac{R}{\m}, M)$ and the result follows. 
 \item If $\mu^d(\p, H^{c- 1}_{\fa}(M))= 0$ then $E_{2}^{d, c- 1}= 0$. Therefore, 
 \[E_{2}^{d- 2, c }= E_{3}^{d- 2, c }= E_{\infty}^{d- 2, c }.\]
 Again,   using the concept of convergence of spectral sequences, $E_{\infty}^{d- 2, c}$ is a subquotient of $\Ext^{d+ c- 2}_{R}(\tfrac{R}{\m}, M)$ and this proves the claim.
 \end{enumerate}
 \end{proof}
As a consequence of the above theorem, we study in the following corollary the  injective dimension $\id_{R}(H^i_{\fa}(M))$ in the case where $R$ is regular. 

 \begin{cor}\label{inj}
Let $R$ be a regular ring and  $c:= \cd_{\fa}(M)$. Then the following statements hold.
 \begin{enumerate}
       \item If $c> 0$ then $\id_{R}(H^c_{\fa}(M))< \dim(R)$ and $H^c_{\fa}(M)$ is not finitely generated.
        \item If $c> 1$ then $\id_{R}(H^c_{\fa}(M))< \dim(R)- 1$.
        \item If $c> 2$ and   $\mu^{\hei(\p)}(\p, H^{c- 1}_{\fa}(M))= 0$ for any $\p\in \Spec(R)$, then $\id_{R}(H^c_{\fa}(M))< \dim(R)- 2$.
 \end{enumerate}
     
 \end{cor}
 \begin{proof}
 Since $R$ is regular, for any $\p\in \Spec(R)$, $  \hei(\p)= \pd_{R_{\p}}(\kappa(\p))$.
 \begin{enumerate}
     \item If $c> 0$  then for any $\p\in \Spec(R)$,  $\Ext^{\hei(\p)+ c}_{R_{\p}}(\kappa(\p), M_{\p})= 0$ and by \ref{reg}(\ref{d,c}) 
 \[\mu^{\dim(R)}(\p, H^c_{\fa}(M))= 0,\,\,\text{for all}\,\,\p\in \Spec(R),\]
 hence   $\id_{R}(H^c_{\fa}(M))< \dim(R)$. 
 
 Now assume, to the contrary, that $H^c_{\fa}(M)$ is finitely generated. Then, in view of \cite[Lemma 1, page 154]{mat}, 
 for any $\p\in \Supp_R(H^c_{\fa}(M))$,  $$\Ext_{R_{\p}}^{\hei(\p)}(\kappa(\p), H^c_{\fa}(M)_{\p})\neq 0,$$ this, in view of \ref{reg}(\ref{d,c}), implies that $\mu^{\hei(\p)+ c}(\p, M)\neq 0$, which is a contraction. 
 \item Let $c> 1$ and $\p\in \Spec(R)$. Then 
 \begin{equation}\label{4}
     \Ext^{\hei(\p)+ c- 1}_{R_{\p}}(\kappa(\p), M_{\p})= 0.
 \end{equation} 
  
      If $\hei(\p)= \dim(R)$, then using \ref{4} and \ref{reg}(\ref{d-1,c}), 
     \[\mu^{\dim(R)- 1}(\p, H^c_{\fa}(M))\leq \mu^{\dim(R)+ c- 1}(\p, M)= 0.\]
       If $\hei(\p)= \dim(R)- 1$, then   by \ref{4} and \ref{reg}(\ref{d,c}),
     \[\mu^{\dim(R)- 1}(\p, H^c_{\fa}(M))= \mu^{\dim(R)-1 +c}(\p, M)= 0.\]
       If $\hei(\p)< \dim(R)- 1$, then since $R$ is regular, $\mu^{\dim(R)- 1}(\p, H^c_{\fa}(M))= 0.$
       Therefore, $\mu^{\dim(R)- 1}(\p, H^c_{\fa}(M))= 0$ for all $\p\in \Spec(R)$.

 \item Now, let $c> 2$ and $\p\in \Spec(R)$. Then 
 \begin{equation}\label{5}
     \Ext^{\hei(\p)+ c- 2}_{R_{\p}}(\kappa(\p), M_{\p})= 0.
 \end{equation} 
 
       If $\hei(\p)= \dim(R)$,   then  by the assumption that  $\mu^{\hei(\p)}(\p, H^{c- 1}_{\fa}(M))= 0$ and   using \ref{5} and \ref{reg}(\ref{d,c-1}), 
     \[\mu^{\dim(R)- 2}(\p, H^c_{\fa}(M))\leq \mu^{\dim(R)+ c- 2}(\p, M)= 0.\]
       When $\hei(\p)= \dim(R)- 1$,     by \ref{5} and \ref{reg}(\ref{d-1,c}),
     \[\mu^{\dim(R)- 2}(\p, H^c_{\fa}(M))\leq \mu^{ \dim(R)+c- 2}(\p, M)= 0.\]
       If $\hei(\p)= \dim(R)- 2$, then by \ref{5} and \ref{reg}(\ref{d,c}),
    \[\mu^{\dim(R)- 2}(\p, H^c_{\fa}(M))=\mu^{ \dim(R)+c- 2}(\p, M)= 0.\] 
       When,  $\hei(\p)< \dim(R)- 2$,  since $R$ is regular, we have $\mu^{\dim(R)- 2}(\p, H^c_{\fa}(M))= 0.$
      Consequently, $\mu^{\dim(R)- 2}(\p, H^c_{\fa}(M))= 0$ for all $\p\in \Spec(R)$.
 Hence, $\id_{R}(H^c_{\fa}(M))< \dim(R)- 2.$
  \end{enumerate}
\end{proof}
  Huneke and   Sharp (\cite{hu}), and independently   Lyubeznik (\cite{lyu2}), showed that if 
if $(R, \m)$ is a regular local ring containing a field then \[\id_{R}(H^i_{\fa}(R))\leq \dim(H^i_{\fa}(R)).\]
Moreover, in \cite{he} Hellus proved that equality \[\id_{R}(H^i_{\fa}(R))= \dim(H^i_{\fa}(R))\] holds in certain cases when 
$R$ is a regular local ring containing a field.
 
 Note that, by the Flat Base Change Theorem (\cite[Theorem 4. 3. 2]{bsh}) together with  Grothendieck’s Vanishing Theorem (\cite[Theorem . 1. 2]{bsh}), one obtains \[\dim(H^i_{\fa}(M))\leq \dim(R)- i,\] whenever $R$  is a Cohen-Macaulay ring. 
 
Now, In view of the above discussion and the preceding corollary, it is natural to ask the following question.
\begin{qu}
     Let 
$R$ be a regular ring. When does the inequality \[\id_{R}(H^i_{\fa}(M))\leq   \dim(R)- i\] hold?
\end{qu}
Note that even in the non-regular case there are some   bounds for $\id_{R}(H^i_{\fa}(M))$ in terms of $\id_{R}(M)$. For example, it is shown in \cite{va}  that if $(R, \m)$ is local and $t\in \NN_0$ such that $H^i_{\fa}(M)=0$ for all $i\neq t$ then $\id_{R}(H^t_{\fa}(M))= \id_{R}(M)- t.$


\begin{thebibliography}{99}


\bibitem{al}
J. Ãlvarez Montaner, F. Sohrabi,  {\em Bass numbers of local cohomology of cover ideals of graphs}, J. Alg. Comb. {\bf53} (2021)  263–297, https://doi.org/10.1007/s10801-019-00928-0.

\bibitem{ba}
N. Abazari and K. Bahmanpour, {\em On the ﬁniteness of Bass numbers of local cohomology modules}, J.
Alg. Appl. 10 (2011)  783-791.

\bibitem{bns}
	K. Bahmanpour, R. Naghipour, M. Sedghi, {\em
On the finiteness of Bass numbers of local cohomology modules and cominimaxness},
Houston J. Math. 40 (2014), 319-337.

 \bibitem{lash}
M. P. Brodmann and A. Lashgari Faghani, {\em A finiteness result for associated primes of
local cohomology modules}, Proc.  
Amer. Math. Soc.,
{\bf128}:10 (2000) 2851–2853

\bibitem{bsh}
M.\ Brodmann,  R.Y.\ Sharp, {\em Local Cohomology: An Algebric
Introduction with Geometric Applications}, 2nd ed. Cambridge University Press, (2012).

\bibitem{bh}
W.\ Bruns and J.\ Herzog, {\em Cohen--Macaulay rings}, 2nd ed. Cambridge University Press, (1998).

 
\bibitem{dnt}
K. Divaani-Aazar, R. Naghipour and M. Tousi, {\em Cohomological dimension of
certain algebraic varieties}, Proc. Amer. Math. Soc.
{\bf 130}(12) (2002) 3537-3544.

\bibitem{ha}
R. Hartshorne, {\em Affine Duality and Coﬁniteness}, Inventiones Mathematicae, {\bf 9}:2
 (1970) 145–64,  https://doi.org/10.1007/BF01404554.

\bibitem{he}
 M. Hellus,  {\em A Note on the injective dimension
of local cohomology modules}, Proc. Amer. Math. Soc.,   \textbf{136}:7  (2008) 2313–2321.
\bibitem{hu1}
C. Huneke, {\em Problems on local cohomology}, Free resolutions in commutative algebra and algebraic
geometry , Res. Notes Math. 2 (1992), 93-108.

\bibitem{hu}
 C. Huneke and R. Y. Sharp, {\em Bass numbers of local cohomology modules,} Trans. Amer. Math. Soc. 339 (1993) 765-779.

 

\bibitem{le}
 A. J. Soto Levins, {\em Bass numbers of the first nonzero local cohomology
Module}, https://doi.org/10.48550/arXiv.2402.15667.


\bibitem{lyu2}
G. Lyubeznik, {\em Finiteness properties of local cohomology modules (an application of D-modules to
commutative algebra)}, Invent. Math., 113 (1993)  41-55.

\bibitem{l}
 G. Lyubeznik, {\em On the vanishing of local cohomology in characteristic $p> 0$}, Compositio Math., {\bf142} (2006) 207-221.

\bibitem{mat}
H. Matsumura, {\em Commutative ring theory}, Cambridge Studies in Advanced
Mathematics 8 (Cambridge University Press, Cambridge, 1986).

\bibitem{o}
A. Ogus, {\em Local cohomological dimensional of algebraic varieties}, Ann. Math., {\bf 98}:2
(1973) 327-365.

\bibitem{r}
J. J. Rotman, {\em An introduction to homological algebra}, 2nd ed. Academic press (2008).


 \bibitem{va}
 A. Vahidi, {\em Injective Dimension of
Local Cohomology Modules,} Bull. Korean Math. Soc., \textbf{54}: 4 (2017)  1331–1336.
https://doi.org/10.4134/BKMS.b160571.


\bibitem{wis} R. Wisbauer, {\em Founations of Module and Ring Theory}, Algebra, Logic and Applications, 3,
Amsterdam: Gordon and Breach (1991).

\bibitem{yan}
 K. Yanagawa, {\em Bass numbers of local cohomology modules with supports in monomial ideals}, Math. Proc.
Cambridge Philos. Soc., {\bf131} (2001) 45-60.
\end{thebibliography}
\end{document}